\documentclass[12pt]{amsart}
\usepackage{amsfonts}
\usepackage{amssymb}
\usepackage{enumerate}
\usepackage{graphicx}

\makeatletter
\@namedef{subjclassname@2010}{  \textup{2010} Mathematics Subject Classification}
\makeatother
\newtheorem{theorem}{Theorem}[section]
\newtheorem{lemma}[theorem]{Lemma}

\newtheorem{corollary}[theorem]{Corollary}

\newcommand{\thistheoremname}{}
\newtheorem*{genericthm*}{\thistheoremname}
\newenvironment{namedthm*}[1]
  {\renewcommand{\thistheoremname}{#1}
  \begin{genericthm*}}
  {\end{genericthm*}}
\theoremstyle{definition}
\newtheorem{definition}[theorem]{Definition}

\newtheorem{quest}{Question}

\newtheorem*{ack}{Acknowledgement}

\numberwithin{equation}{section}
\DeclareMathOperator{\diam}{diam}

\DeclareMathOperator{\co}{co}

\begin{document}
\title[The nonlinear Ryll-Nardzewski theorem]{Around the nonlinear
Ryll-Nardzewski theorem}
\author[A. Wi\'{s}nicki]{Andrzej Wi\'{s}nicki}
\address{Department of Mathematics, Pedagogical University of Krakow,
PL-30-084 Cracow, Poland}
\email{andrzej.wisnicki@up.krakow.pl}
\date{}

\begin{abstract}
Suppose that $Q$ is a weak$^{\ast }$ compact convex subset of a dual Banach
space with the Radon--Nikod\'{y}m property. We show that if $(S,Q)$ is a
nonexpansive and norm-distal dynamical system, then there is a fixed point
of $S$ in $Q$ and the set of fixed points is a nonexpansive retract of $Q.$
As a consequence we obtain a nonlinear extension of the
Bader--Gelander--Monod theorem concerning isometries in $L$-embedded Banach
spaces. A similar statement is proved for weakly compact convex subsets of a
locally convex space, thus giving the nonlinear counterpart of the
Ryll-Nardzewski theorem.
\end{abstract}

\subjclass[2010]{Primary 37B05; Secondary 28D05, 47H10, 54H20}
\keywords{}
\maketitle


\baselineskip=16pt



\section{Introduction}

Fixed point theorems for groups and semigroups of mappings provide a
powerful tool in diverse branches of mathematics. Kakutani-type theorems
have found numerous applications in functional analysis, harmonic analysis
and ergodic theory. Furstenberg's structure theorem and its consequences are
fundamental tools in topological dynamics. The Bruhat--Tits theorem
concerning complete metric spaces satisfying the parallelogram law turns out
to be useful in differential geometry. Kazhdan's property (T) plays a
prominent role in geometric group theory and related fields. Also, the
recent Bader--Gelander--Monod theorem for affine isometries preserving a
bounded set in $L$-embedded Banach spaces has several applications in group
theory and in the theory of operator algebras, including the optimal
solution to the old \textquotedblleft derivation problem\textquotedblright\
in $L^{1}(G).$

One of the first general fixed point theorems for isometries, next to
Kakutani-type theorems, was shown by Brodski\u{\i} and Mil'man in \cite{BrMi}%
--all surjective isometries acting on a weakly compact, convex subset of a
Banach space with normal structure have a common fixed point. In particular,
a group of isometries of a uniformly convex or a uniformly smooth Banach
space with bounded orbits has a fixed point. Since surjective isometries
acting on the whole space are always affine, the last result holds true in
all reflexive spaces as a consequence of the Ryll-Nardzewski theorem, and
still more generally, in all duals of Asplund spaces, thus in all separable
dual spaces provided the isometries are weak$^{\ast }$ continuous. But the
Ryll-Nardzewski theorem goes beyond isometries and concerns a significant
generalization of \textit{isometric} dynamical systems (as defined after
Theorem \ref{ryll}) called \textit{distal} systems. Let us recall its
classical version (see \cite{Ry0, Ry}).

\begin{theorem}[Ryll-Nardzewski]
\label{ryll}Let $Q$ be a nonempty weakly compact convex subset of a locally
convex space $(X,\tau )$ and let $(S,Q)$ be an affine and $\tau $-distal
dynamical system. Then there is a fixed point of $S$ in $Q.$
\end{theorem}

By a dynamical system we mean a pair $(S,Q)$, where $S$ is a semigroup, $Q$
a topological space and there is a semigroup action $S\times Q\rightarrow Q$
of $S$ on $Q$ such that the mappings $Q\ni x\rightarrow sx\in Q$ are
continuous (in the topology of $Q$) for each $s\in S.$ If $X$ is a locally
convex space whose topology $\tau $ is determined by a family $\mathcal{N}$
of seminorms on $X$ and $Q\subset X$, a dynamical system $(S,Q)$ is said to
be isometric if $p(sx-sy)=p(x-y)$ for every $p\in \mathcal{N},s\in S$ and $%
x,y\in Q.$ A dynamical system is affine if $Q$ is convex and $s(\alpha
x+(1-\alpha )y)=\alpha sx+(1-\alpha )sy$ for every $\alpha \in \lbrack
0,1],s\in S$ and $x,y\in Q.$ A dynamical system is compact if $Q$ is
compact, and $\tau $-distal (or $\tau $-noncontracting) if
\begin{equation*}
0\notin \overline{\left\{ sx-sy:s\in S\right\} }^{\tau }
\end{equation*}%
for every pair of distinct points $x,y\in Q.$

Ryll-Nardzewski's theorem involves the interplay between the strong and weak
topology, and considerably improves Hahn's and Kakutani's fixed point
theorems. In particular, since isometric systems are distal and $S$ is a
semigroup, it follows that all continuous affine isometries, not necessarily
surjections, defined on a weakly compact convex set $Q$ have a common fixed
point. The question naturally arises as to whether there is a nonlinear
counterpart of this result or, more generally, whether there exists a
nonlinear version of Ryll-Nardzewski's theorem.

Two remarks are in order. Note first that Alspach's example \cite{Al} shows
that there is a fixed-point free isometry on a weakly compact convex subset
of $L_{1}[0,1].$ Thus, the assumption that the mappings $Q\ni x\rightarrow
sx\in Q$ are weakly continuous cannot be relaxed in general. (Another choice
is to put some additional requirements on $Q$ but it is another interesting
story). Secondly, Boyce \cite{Bo} showed an example of two commuting
continuous maps of $[0,1]$ into itself without a common fixed point, and
Huneke \cite{Hu} showed another such example with Lipschitzian mappings.
However, DeMarr \cite{De} was able to prove that a commutative family of
nonexpansive (i.e., $1$-Lipschitz) mappings defined on a convex compact
subset of a Banach space has a common fixed point. Thus a natural
requirement in the nonlinear case is the nonexpansivity of a dynamical
system and the program of studying fixed point properties of semigroups of
nonexpansive mappings has already been performed in the context of amenable
semigroups (see \cite{LaZh2} and references therein).

In this paper we show the nonlinear version of Ryll-Nardzewski's theorem and
some of its generalizations (see \cite{GlMe} for a thorough study of this
problem). If $Q$ is a subset of a locally convex space whose topology $\tau $
is determined by a family $\mathcal{N}$ of seminorms, a dynamical system $%
(S,Q)$ is said to be $\mathcal{N}$-\textit{nonexpansive} (or briefly,
\textit{nonexpansive} if $\mathcal{N}$ is fixed) if $p(sx-sy)\leq p(x-y)$
for every $p\in \mathcal{N},s\in S$ and $x,y\in Q.$ Here is a summary of our
main results.

\begin{namedthm*}{Theorem A}
Let $Q$ be a weak$^{\ast }$ compact convex subset with the Radon--Nikod\'{y}%
m property of a dual Banach space and let $(S,Q)$ be a nonexpansive and
norm-distal dynamical system. Then there is a fixed point of $S$ in $Q$.
Moreover, the set of fixed points is a nonexpansive retract of $Q$.
\end{namedthm*}

In particular, Theorem A is valid if $Q$ is a weakly compact convex subset
of any Banach space or a norm-separable weak$^{\ast }$ compact convex subset
of a dual Banach space. Furthermore, if $Q$ is a weak$^{\ast }$ compact
convex subset of the dual of an Asplund space, we obtain the nonlinear
version of Namioka--Phelps' theorem (see \cite[Theorem 15]{NaPh}). In
locally convex spaces, we have the following counterpart of the
Ryll-Nardzewski Theorem.

\begin{namedthm*}{Theorem B}
Let $Q$ be a nonempty weakly compact convex subset of a locally convex space
$(X,\tau )$ and let $(S,Q)$ be a nonexpansive and $\tau $-distal dynamical
system. Then there is a fixed point of $S$ in $Q.$ Moreover, the set of
fixed points is a nonexpansive retract of $Q.$
\end{namedthm*}

The qualitative part of both Theorems A and B is a consequence of Bruck's
theorem \cite[Theorem 3]{Br}. Next, we have the following nonlinear
extension of the Bader--Gelander--Monod theorem \cite[Theorem A]{BGM}.

\begin{namedthm*}{Theorem C}
Let $X$ be an $L$-embedded Banach space and let $(S,(X,\mathrm{weak}))$ be a
nonexpansive and norm-distal dynamical system. If there is a bounded set $%
A\subset X$ such that $s(A)=A$ for all $s\in S$, then there is a fixed point
of $S$ located in the Chebyshev center of $A.$
\end{namedthm*}

Theorem C follows from the more general Theorem \ref{L-em} in Section 5. In
particular, it holds for a semigroup of weakly continuous isometries
preserving $A$. Note that in \cite{BGM} a similar statement is asserted for
a group of affine isometries.

The organization of the paper is as follows. In Section 2 we collect the
basic tools that we shall use throughout the proofs, including the
fundamental for our purpose Furstenberg's fixed point theorem. In Section 3
we show a general nonlinear Ryll-Nardzewski type theorem in Banach spaces
and its consequences. Section 4 presents the proof of the first part of
Theorem B. In Section 5 we apply the previous results to prove the
qualitative parts of Theorems A, B and an extension of Theorem C to
dynamical systems in $L$-embedded sets. The paper is concluded with a
nonlinear generalization of Fan's theorem \cite[Theorem 1]{Fa} concerning
the orbits of semigroups of linear contractions in the case of reflexive
Banach spaces.

\section{Main tools}

In this section we list the main components that we shall use to prove the
nonlinear Ryll-Nardzewski theorem and its generalizations.

A bounded subset $A$ of a Banach space $X$ is called dentable if for each $%
\epsilon >0$ there is $x\in A$ such that $x\notin \overline{%
\co%
}(A\setminus B(x,\varepsilon ))$. We recall the following characterization
of a set with the Radon--Nikod\'{y}m property (RNP for short): a bounded
closed convex set $C\subset X$ has the RNP iff every bounded nonempty subset
of $C$ is dentable. A Banach space $X$ is said to have the RNP if its unit
ball has the RNP. It is well known that any weakly compact convex subset of
a Banach space has the RNP as well as any norm-separable weak$^{\ast }$
compact convex subset of a dual space. Moreover, a Banach space $X$ is
Asplund iff $X^{\ast }$ has the RNP.

The notion of dentability is closely related with the concept of
fragmentability, invented by Jayne and Rogers \cite{JaRo}, that is crucial
for the geometric approach to Ryll-Nardzewski's theorem and its
generalizations. The related ideas go back to the works of Glasner \cite{Gl}%
, Hansel and Troallic \cite{HaTr}, Namioka and Phelps \cite{NaPh}, Veech
\cite{Ve}, and culminate in a very general Lemma 1.2 of \cite{GlMe} allowing
one to lift the $\tau $-distality of $(S,Q)$ to the original, usually weaker
topology. Let $(X,\omega )$ be a topological space and let $\rho $ be a
metric on $X$. We say that $(X,\omega )$ is $\rho $-fragmented if for every $%
\varepsilon >0$ and a nonempty set $A\subset X$ there is an $\omega $-open
set $U$ in $X$ such that $U\cap A\neq \emptyset $ and $\rho $-$%
\diam%
(U\cap A)<\varepsilon $. A rather straightforward application of the Baire
category theorem shows that a compact space $(X,\omega )$ is $\rho $%
-fragmented iff for each $\omega $-closed subset $A$ of $X$, the identity
map $\mathrm{id}:(A,\omega )\rightarrow (A,\rho )$ has a point of
continuity. Thus every weak$^{\ast }$ compact convex set with the RNP is
norm-fragmented (see, e.g., \cite[Theorem 4.2.13]{Bou}).

The fundamental tool for our results is the following consequence of
Furstenberg's structure theorem \cite{Fu}, extended from metrizable to
arbitrary compact affine systems by Namioka (see \cite[Theorem 4.1]{Na1},
\cite[Theorem 4.1]{Na4}).

\begin{theorem}[Furstenberg's fixed point theorem]
\label{Furst}Let $(S,Q)$ be a compact affine dynamical system. Suppose there
exists a nonempty compact $S$-invariant subset $K$ (i.e., $s(K)\subset K$
for each $s\in S$) of $Q$ such that $(S,K)$ is distal. Then there is a
common fixed point of $S$ in $Q.$
\end{theorem}

It follows from Theorem \ref{Furst} that every distal compact dynamical
system admits an invariant Radon probability measure. The Radon-Nikodym
property implies that every Radon measure on a weak$^{\ast }$ compact convex
set with the RNP has norm-separable support (see, e.g., \cite[Theorem 4.3.11]%
{Bou}).

The next component is the following observation of DeMarr \cite[Lemma 1]{De}
that also follows from the characterization of normal structure in \cite%
{BrMi}. Recall that a point $x$ of a bounded set $A\subset X$ is called
\textit{diametral} if $\sup_{y\in A}\left\Vert x-y\right\Vert =%
\diam%
A.$ A convex set $K\subset X$ is said to have \textit{normal structur}e if
each bounded, convex subset $A$ of $K$ with $%
\diam%
A>0$ contains a nondiametral point.

\begin{lemma}
\label{DeMarr}Let $X$ be a Banach space and let $K$ be a nonempty compact
subset of $X.$ Then there exists $u\in \overline{%
\co%
}K$ such that $\sup \{\left\Vert x-u\right\Vert :x\in K\}<%
\diam%
K$ provided $%
\diam%
K>0.$
\end{lemma}

It follows that every compact convex subset of $X$ has normal structure.
Note that the result remains unchanged if we replace $X$ by any locally
convex space and a norm by a continuous seminorm.

The link between the weak, weak$^{\ast }$ and norm compactness is given by
the method developed in the nonlinear case by Hsu, To-Ming Lau and Takahashi
(see \cite{Hs}, \cite[Lemma 5.2]{LaTa}), and Bartoszek \cite[Lemma 1]{Ba}.
We reformulate the result of Hsu in the case of locally convex spaces. Let $%
X $ be a locally convex space whose topology $\tau $ is determined by a
family $\mathcal{N}$ of seminorms on $X$ and let $K\subset X$. We say that $%
K $ is a minimal weakly compact $S$-invariant subset of $X$ if there is no
proper (nonempty) weakly compact $S$-invariant set $K_{0}\subsetneq K.$

\begin{lemma}
\label{Hsu}Let $(S,K)$ be an $\mathcal{N}$-nonexpansive dynamical system,
where $K$ is a minimal weakly compact $S$-invariant and $\tau $-separable
subset of a locally convex space $(X,\tau )$ such that $s(K)=K$ for each $%
s\in S.$ Then $K$ is $\tau $-totally bounded.
\end{lemma}

\begin{proof}
Let $U=\{x\in X:p_{1}(x)<\varepsilon ,...,p_{n}(x)<\varepsilon
\},p_{1},...,p_{n}\in \mathcal{N},\varepsilon >0,$ be a $\tau $-open
neighbourhood of $0$ and take a convex $\tau $-closed neighbourhood $V$ of $%
0 $ such that $V-V\subset U.$ Since $K$ is $\tau $-separable, there exists a
sequence $\{x_{n}\}\subset K$ such that $K\subset \bigcup_{n=1}^{\infty
}(x_{n}+V).$ Since $(K,\mathrm{weak})$ is a Baire space and each translate
of $V$ is weakly closed, there is a weakly open neighbourhood $W$ of $0$ and
$z,x_{i}\in K$ such that%
\begin{equation*}
(z+W)\cap K\subset (x_{i}+V)\cap K\subset (z+U)\cap K.
\end{equation*}%
Take a weakly open neighbourhood $W^{\prime }$ of $0$ such that $W^{\prime
}+W^{\prime }\subset W$ and a $\tau $-open neighbourhood of $0$, $U^{\prime
}=\{x\in X:q_{1}(x)<\varepsilon ^{\prime },...,q_{m}(x)<\varepsilon ^{\prime
}\}$, $q_{1},...,q_{m}\in \mathcal{N},\varepsilon ^{\prime }>0$, such that $%
U^{\prime }\subset W^{\prime }$. By $\tau $-separability, there is a
sequence $\{y_{n}\}\subset K$ such that $K\subset \bigcup_{n=1}^{\infty
}(y_{n}+U^{\prime }).$ Since $K$ is minimal $S$-invariant, $\{sy:s\in S\}$
is weakly dense in $K$ for each $y\in K,$ and hence we can choose by
induction a sequence $\{s_{n}\}\subset S$ such that $%
s_{1}y_{1},s_{2}s_{1}y_{2},...,s_{n}s_{n-1}...s_{1}y_{n},...\in z+W^{\prime
}.$ It follows from $\mathcal{N}$-nonexpansivity that%
\begin{equation*}
s_{n}s_{n-1}...s_{1}((y_{n}+U^{\prime })\cap K)\subset
s_{n}s_{n-1}...s_{1}y_{n}+U^{\prime }\subset z+W
\end{equation*}%
for each $n,$ and thus $K\subset \bigcup_{n=1}^{\infty
}(s_{n}s_{n-1}...s_{1})^{-1}(z+W).$ Since $K$ is weakly compact, there
exists a finite subcover $K\subset
\bigcup_{n=1}^{p}(s_{n}s_{n-1}...s_{1})^{-1}(z+W).$ Now we have\
\begin{eqnarray*}
K &=&s_{p}s_{p-1}...s_{1}(K)\subset
\bigcup_{n=2}^{p}s_{p}s_{p-1}...s_{n}((z+W)\cap K)\cup ((z+W)\cap K) \\
&\subset &\bigcup_{n=2}^{p}s_{p}s_{p-1}...s_{n}((z+U)\cap K)\cup ((z+U)\cap
K)
\end{eqnarray*}%
and from the nonexpansivity, $K\subset
\bigcup_{n=2}^{p}(s_{p}s_{p-1}...s_{n}z+U)\cup (z+U),$ that is, $K$ is
totally bounded.
\end{proof}

Note that in Banach spaces the above argument works when $K$ is weak$^{\ast
} $ compact too (see \cite[Lemma 5.2]{LaTa}). Another approach to this
problem, adapted in Section 3, was given by Bartoszek \cite[Lemma 1]{Ba}.

In what follows, we show how the results described in this section interact
with one another and with some classical arguments in this field to prove
nonlinear Ryll-Nardzewski type theorems.

\section{Fixed points of distal systems and the Radon-Nikodym property}

In this section we prove a general nonlinear fixed point theorem of
Ryll-Nardzewski type in Banach spaces.

\begin{theorem}
\label{gen}Let $Q$ be a weak$^{\ast }$ compact convex subset with the
Radon-Nikod\'{y}m property of a dual Banach space $X$ and let $(S,Q)$ be a
nonexpansive and norm-distal dynamical system. Then there is a fixed point
of $S$ in $Q$.
\end{theorem}

\begin{proof}
By Kuratowski--Zorn's lemma, we can assume without loss of generality that $%
Q $ is a minimal $S$-invariant weak$^{\ast }$ compact convex subset with the
RNP of $X$. Let $K$ be a minimal $S$-invariant weak$^{\ast }$ compact subset
of $Q.$ We first prove that $(S,K)$ is weak$^{\ast }$-distal. A trick is to
`lift' the distality, using fragmentability. A general result of this type,
inspired by \cite[Proposition 2]{HaTr}, is shown in \cite[Lemma 1.2]{GlMe}.
We present this argument in the case of norm and weak$^{\ast }$ topologies.
Fix $x,y\in K.$ By norm-distality, there is $\varepsilon >0$ such that
\begin{equation}
\left\Vert sx-sy\right\Vert >\varepsilon  \label{distal}
\end{equation}%
for every $s\in S.$ Suppose conversely that $S$ is not weak$^{\ast }$-distal
on $K$, i.e., there are nets $(s_{\alpha }x),(s_{\alpha }y)$ such that $%
w^{\ast }$-$\lim s_{\alpha }x=w^{\ast }$-$\lim s_{\alpha }y=u$ for some $%
u\in K.$ Notice that $K=\overline{\{su:s\in S\}}^{\mathrm{weak}^{\ast }}$ by
the minimality of $K.$ Since $K$ is norm-fragmented as a subset of a weak$%
^{\ast }$ compact convex set $Q$ with the RNP, there is a weak$^{\ast }$%
-open set $U$ in $Q$ such that $U\cap K\neq \emptyset $ and $%
\diam%
(U\cap K)<\varepsilon $. Hence there exists $s_{0}\in S$ such that $%
p=s_{0}u\in U\cap K.$ Then%
\begin{equation*}
w^{\ast }\text{-}\lim s_{0}s_{\alpha }x=p=w^{\ast }\text{-}\lim
s_{0}s_{\alpha }y
\end{equation*}%
and thus $s_{0}s_{\alpha }x,s_{0}s_{\alpha }y\in U\cap K$ eventually. But
this contradicts (\ref{distal}) and $S$ is weak$^{\ast }$-distal on $K.$

Let $C(K)$ denote the space of continuous functions on $K$ and define $%
(f\cdot s)(t)=f(st)$ for $f\in C(K)$ and $s,t\in S.$ Let $T_{s}f=f\cdot s$
for each $s\in S.$ Notice that $(S,P(K))$ is an affine dynamical system with
the action $s\cdot \mu =T_{s}^{\ast }(\mu ),$ where $P(K)$ is the convex weak%
$^{\ast }$-compact set of all means on $C(K)$ (i.e., Radon probability
measures on $K$ with the weak$^{\ast }$ topology) and $T_{s}^{\ast
}:C(K)^{\ast }\rightarrow C(K)^{\ast }$ is the adjoint of $T_{s}.$ Let $\phi
:K\rightarrow P(K)$ be the natural embedding of $K$ into $P(K)$ defined by $%
\phi (x)(f)=f(x).$ Then $\phi $ is an isomorphism of systems $(S,K)$ and $%
(S,\phi (K)).$ Thus $(S,\phi (K))$ is weak$^{\ast }$ distal since $(S,K)$ is
and, by Theorem \ref{Furst}, there is a fixed point $\mu $ of $S$ in $P(K)$,
that is, $\mu $ is an $S$-invariant Radon probability measure on $K$ (with
respect to weak$^{\ast }$ topology).

Define $K_{0}=\mathrm{supp}(\mu )$ and notice that $\mu (s^{-1}(K_{0}))=\mu
(K_{0})=1.$ Furthermore, $s^{-1}(K_{0})$ is weak$^{\ast }$ closed and $K_{0}$
is the least weak$^{\ast }$ closed subset of $K$ of full measure. Hence $%
K_{0}\subset s^{-1}(K_{0}).$ Similarly,
\begin{equation*}
\mu (s(K_{0}))=\mu (s^{-1}(s(K_{0})))=\mu (K_{0})=1
\end{equation*}%
and consequently, from weak$^{\ast }$ compactness of $s(K_{0})$, $%
K_{0}\subset s(K_{0}).$ Thus $s(K_{0})=K_{0}$ for every $s\in S$ and, since $%
K$ is a minimal $S$-invariant weak$^{\ast }$ compact subset of $Q$, $%
K=K_{0}. $ We show that $K$ is norm-compact. Since $K$ is weak$^{\ast }$
compact and $Q$ has the RNP, the identity map $\mathrm{id}:(K,\mathrm{weak}%
^{\ast })\rightarrow (K,\mathrm{norm})$ has a point of continuity $x$ (see,
e.g., \cite[Theorem 4.2.13]{Bou}). It follows that for every $\varepsilon >0$
there is a weak$^{\ast }$ open neighbourhood $U$ of $x$ such that $%
\left\Vert x-y\right\Vert <\varepsilon $ for each $y\in U\cap K.$ But $x\in
\mathrm{supp}(\mu )$ and hence
\begin{equation*}
\mu (\{y\in K:\left\Vert x-y\right\Vert <\varepsilon \})\geq \mu (U\cap K)>0
\end{equation*}%
for each $\varepsilon >0.$

Now we follow partly the argument of Bartoszek \cite[Lemma 1]{Ba}. Fix $%
\varepsilon >0$ and let $\mu (\{y\in K:\left\Vert x-y\right\Vert
<\varepsilon \})=\delta >0.$ Notice that by nonexpansivity,%
\begin{equation*}
\{y\in K:\left\Vert x-y\right\Vert <\varepsilon \}\subset s^{-1}(\{y\in
K:\left\Vert sx-y\right\Vert <\varepsilon \})
\end{equation*}%
for each $s\in S$ and, since $\mu $ is $S$-invariant,%
\begin{equation*}
\mu (\{y\in K:\left\Vert sx-y\right\Vert <\varepsilon \})\geq \mu (\{y\in
K:\left\Vert x-y\right\Vert <\varepsilon \})=\delta .
\end{equation*}%
It follows that a number of elements $s_{1},s_{2},...,s_{k}$ such that%
\begin{equation*}
\left\Vert s_{i}x-s_{j}x\right\Vert \geq 2\varepsilon
\end{equation*}%
for each $1\leq i\neq j\leq k$ is bounded by $1/\delta .$ Hence $\{sx:s\in
S\}$ is norm-totally bounded and therefore, $K=\overline{\{sx:s\in S\}}^{%
\mathrm{weak}^{\ast }}=\overline{\{sx:s\in S\}}^{\mathrm{norm}}$ is
norm-compact.

We show that $K$ consists of a single point. Suppose that $r=\diam K>0.$
Then by Lemma \ref{DeMarr}, there is $u\in \overline{%
\co%
}K$ such that $r_{0}=\sup \{\left\Vert u-y\right\Vert :y\in K\}<r.$ Define
\begin{equation*}
Q_{0}=\{x\in Q:\left\Vert x-y\right\Vert \leq r_{0}\text{ for all }y\in K\}.
\end{equation*}%
Then $u\in Q_{0}$ and $Q_{0}$ is a weak$^{\ast }$ compact convex proper
subset of $Q.$ Since the system $(S,Q)$ is nonexpansive and $s(K)=K$, it
follows that $s(Q_{0})\subset Q_{0}$ for each $s\in S$ which contradicts the
minimality of $Q.$ Thus $\diam K=0$ and $K$ consists of a single point which
is a common fixed point of $S$ in $K.$
\end{proof}

As a consequence, we obtain the nonlinear version of Namioka--Phelps'
theorem \cite[Theorem 15]{NaPh}. Recall that a Banach space $X$ is Asplund
if every continuous convex function on any open convex subset $U$ of $X$ is
Fr\'{e}chet differentiable on a dense $G_{\delta }$-subset of $U$. By the
results of Namioka, Phelps and Stegall, $X$ is Asplund iff $X^{\ast }$ has
the RNP.

\begin{corollary}
\label{asp}Suppose $X$ is an Asplund space and $Q$ a weak$^{\ast }$ compact
convex subset of $X^{\ast }$. If $(S,Q)$ is a nonexpansive and norm-distal
dynamical system, then there is a fixed point of $S$ in $Q$.
\end{corollary}

Similarly, since every separable weak$^{\ast }$ compact convex subset of a
dual space has the RNP, we have

\begin{corollary}
Let $Q$ be a separable weak$^{\ast }$ compact convex subset of a dual Banach
space $X^{\ast }$ and let $(S,Q)$ be a nonexpansive and norm-distal
dynamical system. Then there is a fixed point of $S$ in $Q$.
\end{corollary}

Furthermore, every weakly compact convex subset of a Banach space $X$ can be
regarded as a weak$^{\ast }$ compact convex subset (with the RNP) of $%
X^{\ast \ast }.$ Hence we obtain the nonlinear Ryll-Nardzewski theorem in a
Banach space.

\begin{corollary}
\label{weakly}Let $Q$ be a weakly compact convex subset of a Banach space $X$
and let $(S,Q)$ be a nonexpansive and norm-distal dynamical system. Then
there is a fixed point of $S$ in $Q$.
\end{corollary}

In the next section we generalize Corollary \ref{weakly} to locally convex
spaces.

\section{Nonlinear Ryll-Nardzewski's theorem in locally convex spaces}

In a locally convex space $(X,\tau ),$ the arguments in the proof of Theorem %
\ref{gen} are not completely applicable. However, in the case of weakly
compact convex sets, we can reduce the problem (by a classical argument) to $%
S$ being countable and then use Lemma \ref{Hsu}.

\begin{theorem}
\label{loc}Let $Q$ be a weakly compact convex subset of a locally convex
space $(X,\tau )$ whose topology $\tau $ is determined by a family $\mathcal{%
N}$ of seminorms on $X.$ If a dynamical system $(S,Q)$ is $\mathcal{N}$%
-nonexpansive and $\tau $-distal, then there is a fixed point of $S$ in $Q$.
\end{theorem}

\begin{proof}
By weak compactness, it is sufficient to show that each finite subset of $S$
has a common fixed point in $Q.$ Hence we can assume that $S$ is countable
and $Q$ is a minimal $S$-invariant weakly compact convex subset of $X$. Let $%
K$ be a minimal $S$-invariant weakly compact subset of $Q.$

We show that $S$ is weakly-distal on $K.$ If not, then there are $x,y,u\in K$
and nets $(s_{\alpha }x),(s_{\alpha }y)$ such that $w$-$\lim s_{\alpha }x=w$-%
$\lim s_{\alpha }y=u$. Notice that $K=\overline{\{su:s\in S\}}^{\mathrm{weak}%
}$ by the minimality of $K.$ Since $(S,K)$ is $\tau $-distal, there is a $%
\tau $-neighborhood $U$ of $0$ such that $s(x)-s(y)\notin U$ for each $s\in
S $. Let $V$ be a convex $\tau $-closed neighborhood of $0$ such that $%
V-V\subset U$. By Mazur's lemma, $K\subset \overline{%
\co%
}^{\tau }\{su:s\in S\}$ and, since $S$ is countable, $K$ is $\tau $%
-separable. Hence there exists a sequence $\{y_{n}\}\subset K$ such that $%
K\subset \bigcup_{n=1}^{\infty }(y_{n}+V).$ From Baire's category theorem,
there is a weakly open neighbourhood $W$ of $0$ and $z,y_{i}\in K$ such that
$(z+W)\cap K\subset y_{i}+V.$ Thus there is $s_{0}\in S$ such that $%
s_{0}u\in z+W.$ It follows that%
\begin{equation*}
w\text{-}\lim s_{0}s_{\alpha }x=s_{0}u=w\text{-}\lim s_{0}s_{\alpha }y
\end{equation*}%
and eventually $s_{0}s_{\alpha }x,s_{0}s_{\alpha }y\in (z+W)\cap K\subset
y_{i}+V.$ Hence $s_{0}s_{\alpha }x-s_{0}s_{\alpha }y$ is eventually in $%
V-V\subset U$ but it contradicts our choice of $U.$ Thus $S$ is
weakly-distal on $K.$

Now, as in the proof of Theorem \ref{gen}, we can show that $K$ admits an $S$%
-invariant Radon probability measure and $s(K)=K$ for each $s\in S.$
Moreover, $K$ is $\tau $-separable and it follows from Lemma \ref{Hsu} that $%
K$ is totally bounded. We can certainly assume that $(X,\tau )$ is complete
(since otherwise we consider the closure of $K$ in the completion of $X$).
Thus $K$ is $\tau $-compact. If $K$ consists of more than one point, then
there exists a seminorm $q\in \mathcal{N}$ such that $r=\sup \{q(x-y):x,y\in
K\}>0.$ Then by a counterpart of Lemma \ref{DeMarr}, there is $u\in
\overline{%
\co%
}K$ such that $r_{0}=\sup \{q(u-y):y\in K\}<r.$ Define
\begin{equation*}
Q_{0}=\{x\in Q:q(x-y)\leq r_{0}\text{ for all }y\in K\}.
\end{equation*}%
Then $u\in Q_{0}$ and $Q_{0}$ is a weakly compact convex proper subset of $%
Q. $ Since the action is $\mathcal{N}$-nonexpansive and $s(K)=K,s\in S,$ we
have $s(Q_{0})\subset Q_{0}$ for each $s\in S$ which contradicts the
minimality of $Q.$ Thus $K$ consists of a single point $x$ and $sx=x$ for
every $s\in S.$
\end{proof}

There is a natural generalization of Asplund Banach spaces, introduced in
\cite[Definition 4.1]{Me} (see also \cite[Definition 1.10]{GlMe}): a locally
convex space $(X,\tau )$ is called a Namioka--Phelps space if every
equicontinuous subset $K$ in $X^{\ast }$ is $(\mathrm{weak}^{\ast },\xi
^{\ast }$)-fragmented, where $\xi ^{\ast }$ denotes the natural uniform
structure of $X^{\ast }$. It is shown in \cite{Me} that the class of
Namioka--Phelps spaces contains in particular Fr\'{e}chet differentiable
spaces, semireflexive spaces and nuclear spaces, and is closed under taking
subspaces, products and direct sums (see also \cite[Remark 1.12]{GlMe}).
Since the weak compactness of $Q$ is applied in the proof of Theorem \ref%
{loc} in a substantial way to show the $\tau $-separability of $K$, it is
not clear how to extend Corollary \ref{asp} to the case of locally convex
spaces. This leads to the following natural questions.

\begin{quest}
Is it true that Corollary \ref{asp} remains true for Namioka--Phelps spaces?
\end{quest}

More generally, we can ask:

\begin{quest}
Do there exist nonlinear counterparts of Theorems 1.5 and 1.6 in \cite{GlMe}?
\end{quest}

\section{Applications}

In 2012, Bader, Gelander and Monod \cite{BGM} gave a beautiful proof of a
fixed point theorem in $L$-embedded Banach spaces. One of its consequences
is the optimal solution to the following \textquotedblleft derivation
problem\textquotedblright : if $G$ is a locally compact group, then any
derivation from the convolution algebra $L^{1}(G)$ to $M(G)$ is inner. The
problem was studied since 1960s and proved for the first time by Losert \cite%
[Theorem 1.1]{Lo}. We apply Theorem \ref{gen} to show a nonlinear extension
of Bader--Gelander--Monod's theorem \cite[Theorem A]{BGM}.

Recall that a Banach space $X$ is said to be $L$-embedded if its bidual $%
X^{\ast \ast }$ can be decomposed as $X^{\ast \ast }=X\oplus _{1}X_{s}$ for
some $X_{s}\subset X^{\ast \ast }$ (with the norm being the sum of norms of $%
X$ and $X_{s}$). The class of $L$-embedded Banach spaces includes all $L_{1}$
spaces, preduals of von Neumann algebras and the Hardy space $H_{1}$. We
need the following generalization.

\begin{definition}[see \protect\cite{LaZh2}]
Let $C$ be a nonempty subset of a Banach space $X$ and denote by $\overline{C%
}^{\mathrm{wk}^{\ast }}$ the closure of $C$ in $X^{\ast \ast }$ in the weak$%
^{\ast }$ topology of $X^{\ast \ast }$. We say that $C$ is $L$-embedded if
there is a subspace $X_{s}$ of $X^{\ast \ast }$ such that $X\oplus
_{1}X_{s}\subset X^{\ast \ast }$ and $\overline{C}^{\mathrm{wk}^{\ast
}}\subset C\oplus _{1}X_{s}.$
\end{definition}

It was proved in \cite{LaZh2} that every $L$-embedded set is weakly closed.
Moreover, a Banach space is $L$-embedded iff its unit ball is $L$-embedded.
Notice that a weakly compact subset $C$ of any Banach space $X$ is $L$%
-embedded since $\overline{C}^{\mathrm{wk}^{\ast }}=C.$

If $A,C$ are subsets of a Banach space $X$ with $A$ bounded, we define the
Chebyshev radius of $A$ in $C$ by%
\begin{equation*}
r_{C}(A)=\inf_{x\in C}\sup_{y\in A}\left\Vert x-y\right\Vert
\end{equation*}%
and the Chebyshev center of $A$ in $C$ by%
\begin{equation*}
E_{C}(A)=\{x\in C:\sup_{y\in A}\left\Vert x-y\right\Vert =r_{C}(A)\}.
\end{equation*}

\begin{lemma}[{see {\protect\cite[Lemma 3.3]{LaZh2}}}]
\label{embed}Let $C$ be an $L$-embedded subset of a Banach space $X$ and $A$
a bounded subset of $X$. Then the Chebyshev center $E_{C}(A)$ is weakly
compact.
\end{lemma}

Combining Corollary \ref{weakly} with Lemma \ref{embed} yields the following
extension of \cite[Theorem A]{BGM}.

\begin{theorem}
\label{L-em}Let $Q$ be a bounded convex $L$-embedded subset of a Banach
space $X$ and let $(S,(Q,\mathrm{weak}))$ be a nonexpansive and norm-distal
dynamical system. If $Q$ contains a bounded subset $A$ such that $s(A)=A$
for all $s\in S$, then there is a fixed point of $S$ in $E_{Q}(A).$
\end{theorem}

\begin{proof}
Notice that $E_{Q}(A)$ is convex, $s(E_{Q}(A))\subset E_{Q}(A)$ and by Lemma %
\ref{embed}, $E_{Q}(A)$ is weakly compact. Therefore, we can apply Corollary %
\ref{weakly} to obtain a fixed point of $S$ in $E_{Q}(A).$
\end{proof}

\begin{corollary}
Let $A$ be a non-empty bounded subset of an $L$-embedded Banach space $X$.
Then there is a point in $X$ fixed by every weakly continuous isometry of $X$
into $X$ preserving $A.$ Moreover, one can choose a fixed point which
minimizes $\sup_{a\in A}\left\Vert x-a\right\Vert $ over all $x\in X.$
\end{corollary}

Note that in \cite{BGM} a similar statement is asserted for affine
isometries.

Our next result concerns the qualitative information about the structure of
the set of fixed points of nonexpansive distal systems. We shall rely on the
following consequence of Bruck's theorem \cite[Theorem 3]{Br}.

\begin{theorem}
\label{Bruck}Let $Q$ be a compact Hausdorff topological space and $S$ a
(discrete) semigroup of mappings on $Q$. Suppose that $S$ is compact in the
topology of pointwise convergence and each nonempty closed $S$-invariant
subset of $Q$ contains a fixed point of $S$. Then there exists in $S$ a
retraction of $Q$ onto $F(S)=\{x\in Q:sx=x$ for every $s\in S\}.$
\end{theorem}

Note that a retraction in the above theorem is simply a mapping $%
r:Q\rightarrow F(S)$ such that $r\circ r=r.$ (The continuity of $r$ in the
topology of $Q$ is not required). The following theorem is the qualitative
part of Theorem A alluded to in the introduction.

\begin{theorem}
Let $Q$ be a weak$^{\ast }$ compact convex subset with the RNP of a dual
Banach space $X$ and let $(S,Q)$ be a nonexpansive and norm-distal dynamical
system. Then the set $F(S)$ of fixed points of $S$ is a nonexpansive retract
of $Q.$
\end{theorem}

\begin{proof}
Put $\hat{S}=\{T:Q\rightarrow Q\mid T$ is nonexpansive and $F(S)\subset
F(T)\}.$ Notice that $Q^{Q}$ is compact in the product topology when $Q$ is
given the weak$^{\ast }$ topology and the product topology is the topology
of weak$^{\ast }$ pointwise convergence. Moreover, $\hat{S}\subset Q^{Q}$ is
closed in this topology since
\begin{equation*}
\left\Vert w^{\ast }\text{-}\lim T_{\alpha }x-w^{\ast }\text{-}\lim
T_{\alpha }y\right\Vert \leq \liminf_{\alpha }\left\Vert T_{\alpha
}x-T_{\alpha }y\right\Vert \leq \left\Vert x-y\right\Vert
\end{equation*}%
for every $x,y\in Q$ and every convergent net $\{T_{\alpha }\}\subset \hat{S}
$, and $w^{\ast }$-$\lim T_{\alpha }x=x$ whenever $sx=x,s\in S.$ Thus $\hat{S%
}$ is a compact semigroup in the topology of weak$^{\ast }$ pointwise
convergence and $F(S)=F(\hat{S}).$ Let $Q_{0}$ be a weak$^{\ast }$ closed $%
\hat{S}$-invariant subset of $Q.$ Choose $x\in Q_{0}$ and notice that $\hat{Q%
}=\{Tx:T\in \hat{S}\}$ is a weak$^{\ast }$ compact $\hat{S}$-invariant
subset of $Q_{0}$. Moreover, $\hat{Q}$ is convex since $\alpha
T_{1}+(1-\alpha )T_{2}\in \hat{S}$ if $T_{1},T_{2}\in \hat{S}$ and $\alpha
\in \lbrack 0,1].$ By Theorem \ref{gen}, there is a fixed point of $\hat{S}$
in $\hat{Q}\subset Q_{0}.$ Now it follows from Theorem \ref{Bruck} that
there exists in $\hat{S}$ a retraction of $Q$ onto $F(\hat{S})=F(S).$ But
every element in $\hat{S}$ is nonexpansive (though, not necessarily weak$%
^{\ast }$ continuous).
\end{proof}

In a similar way we have the qualitative part of Theorem B.

\begin{theorem}
Let $Q$ be a weakly compact convex subset of a locally convex space $(X,\tau
)$ and let $(S,Q)$ be an $\mathcal{N}$-nonexpansive and $\tau $-distal
dynamical system. Then the set $F(S)$ of fixed points of $S$ is an $\mathcal{%
N}$-nonexpansive retract of $Q.$
\end{theorem}

We end the paper with the following nonlinear extension of Fan's result \cite%
[Theorem 1]{Fa} in reflexive spaces. Let $X$ be a Banach space and $E$ be a
subset of $X^{\ast }$. We will say that a dynamical system $(S,E)$ is
\textit{deflating} if there exist distinct $\varphi _{1},\varphi _{2}\in E$
such that for every absolutely convex, weakly compact set $C\subset X$ and
any $\varepsilon >0$, there is $s\in S$ such that $\left\vert (s\varphi
_{1})(y)-(s\varphi _{2})(y)\right\vert <\varepsilon $ for every $y\in C.$

\begin{theorem}
\label{Fan}Let $X$ be a reflexive Banach space, $z\in X,\left\Vert
z\right\Vert =1$ and $Q=\{\varphi \in X^{\ast }:\left\Vert \varphi
\right\Vert =\varphi (z)=1\}\subset E.$ Let $(S,(Q,\mathrm{weak}))$ be a
non-deflating and norm-nonexpansive dynamical system. Then there exists $%
\psi \in Q$ such that $s\psi =\psi $ for all $s\in S.$
\end{theorem}

\begin{proof}
It is clear that $Q$ is a convex weakly compact subset of $X^{\ast }.$ Let $%
\varphi _{1},\varphi _{2}$ be two distinct elements of $Q.$ Since $S$ is
non-deflating, there is a weakly compact absolutely convex subset $C$ of $X$
and $\varepsilon >0$ such that for every $S\in S$ there is $y_{0}\in C$
satisfying $\left\vert (s\varphi _{1})(y_{0})-(s\varphi
_{2})(y_{0})\right\vert \geq \varepsilon .$ Thus $\{\varphi \in X^{\ast
}:\left\vert \varphi (y)\right\vert <\varepsilon $ for $y\in C\}$ is a $\tau
(X^{\ast },X)$-neighbourhood of $0$ with respect to the Mackey topology $%
\tau (X^{\ast },X)$ of $X^{\ast }$ which is disjoint from the set $%
\{s\varphi _{1}-s\varphi _{2}:s\in S\}.$ Since $X$ is reflexive, $\tau
(X^{\ast },X)$ coincides with the norm topology and thus $(S,Q)$ is
norm-distal. Now the thesis follows from Corollary \ref{weakly}.
\end{proof}

Let $z\in X,\left\Vert z\right\Vert =1$, and let $S$ be a semigroup of
linear mappings $s:X\rightarrow X$ such that $\left\Vert u\right\Vert \leq 1$
and $s(z)=z$ for each $s\in S.$ Suppose that $S$ has no direction of
deflation, i.e., for every $0\neq \varphi \in X^{\ast }$ there is an
absolutely convex, weakly compact subset of $X$ and $\varepsilon >0$ such
that $u(C)\not\subset \{{y\in X:\left\vert \varphi (y)\right\vert
<\varepsilon \}.}$ Then the semigroup $S^{\ast }=\{s^{\ast }:X^{\ast
}\rightarrow X^{\ast }\mid s\in S\}$ of adjoints of $S$ satisfies the
assumptions of Theorem \ref{Fan} since it is non-deflating,
norm-nonexpansive and $(s^{\ast }\varphi )(z)=\varphi (s(z))=\varphi (z)=1$
for every $\varphi \in Q$ and $s\in S.$ Hence there exists $\psi \in Q$ such
that $\psi (s(x))=\psi (x)$ for all $x\in X$ and $s\in S.$ Thus Theorem \ref%
{Fan} is a nonlinear generalization of \cite[Theorem 1]{Fa} in the case of
reflexive spaces.

\begin{ack}
The author is grateful to the referee whose valuable comments have improved
the presentation of the paper.
\end{ack}

\end{document}